\documentclass[3p,12pt]{elsarticle}
\usepackage{amsmath,amssymb,amsfonts,verbatim}
\journal{Journal of Algebra}

\usepackage{shapiro}
\newtheorem{theorem}{Theorem}[section]
\newtheorem{lemma}[theorem]{Lemma}
\newtheorem{prop}[theorem]{Proposition}
\newtheorem{corollary}[theorem]{Corollary}

\newenvironment{proof}{{\flushleft\it Proof.}}{\hfill$\square$ \\}
\newenvironment{remark}{{\flushleft\it Remark\ts.}}{\hfill$\square$ \\}
\numberwithin{equation}{section}

\def\beq{\begin{equation}}
\def\eeq{\end{equation}}
\def\beqs{\begin{equation*}}
\def\eeqs{\end{equation*}}

\def\A{\mathrm{A}}
\def\Ab{\ts\bar{\A}\ts}
\def\ad{\operatorname{ad}}
\def\al{G}

\def\Aut{\operatorname{Aut}}
\def\aw{\la}

\def\Bb{\mathrm{B}}

\def\Cbb{\mathbb{C}}
\def\CC{\mathbb{C}}
\def\comp{\ts{\scriptstyle\circ}\ts}

\def\d{\partial}
\def\D{{\mathcal D}}
\def\de{\delta}
\def\De{\Delta}
\def\Dep{\De^{\,\prime}}

\def\E{D}
\def\EE{\E^{\,\sign}}

\def\ev{\pi}
\def\evp{\pi^{\,\prime}}

\def\G{{\mathcal G}}

\def\GL{\operatorname{GL}}

\def\gl{\mathfrak{gl}}
\def\g{\mathfrak{g}}

\def\h{\mathfrak{h}}

\def\I{{\operatorname I}}
\def\Ib{{\bar\I}}
\def\Id{\operatorname{Id}}
\def\Img{\operatorname{Im}}
 
\def\Ipb{\Ib{}^{\,\prime}} 
\def\iso{A\ts}

\def\J{{\operatorname J}}
\def\Jb{{\bar \J}}
 
\def\Jpb{\Jb{}^{\,\prime}} 

\def\La{\Lambda}
\def\la{\lambda}
\def\lap{\bar\lambda}

\def\M{\Phi}
\def\Mlm{\M^{\ts\la}_{\ts\mu}}
\def\mup{\mu^{\ts\prime}}

\def\Nb{\bar{N}}
\def\nup{\bar\nu}

\def\om{\omega}

\def\n{\mathfrak{n}}

\def\Pc{\mathcal{P}}
\def\phi{\varphi}
\def\zec{\check{\zeta}_c}
\def\zeclm{{\check{\zeta}_{c,\mu}^\lambda}}

\def\rf#1{(\ref{#1})}

\def\S{\mathfrak{S}}%{\operatorname{S}}
\def\si{\sigma}

\def\sign{\varepsilon}

\def\ts{\hskip.75pt}

\def\U{\mathrm{U}}
\def\Uh{\ts\bar{\U}(\h)\ts}

\def\vlm{v^{\ts\la}_{\ts\mu}}
\def\V{\operatorname{V}}

\def\xic{\xi}
\def\xiclm{U_a}

\def\Y{\mathrm{Y}}
\def\YY{\Y(\gl_n)}

\def\ZZ{\mathbb{Z}}

\def\zec{\eta}
\def\zeclm{V_a}

\begin{document}

\begin{frontmatter}

\title{Rational and polynomial representations of Yangians}
 
\author[a,b]{Sergey Khoroshkin\,}
\address[a]{Institute for Theoretical and Experimental Physics, 
Moscow 117259, Russia}
\address[b]{Laboratory of Mathematical Physics, Higher School of Economics,
Moscow 117312, Russia}
%\ead{khor@itep.ru}

\author[c]{Maxim Nazarov\,}
\address[c]{Department of Mathematics,  University of York, 
York YO10 5DD, England}
%\ead{mln1@york.ac.uk}

\author[a,d]{Alexander Shapiro\,}
\address[d]{Department of Mathematics, University of California, 
Berkeley 94720, California}
%\ead{alexander.m.shapiro@gmail.com}


\begin{abstract}
We define natural classes of rational and polynomial
representations of the Yangian of the general linear Lie
algebra. We also present the classification and
explicit realizations of all irreducible rational
representations of the Yangian.
\end{abstract}

\end{frontmatter}

\thispagestyle{empty}%%%%%%%%%%%%%%%%%%%%%%%%%%%%%%%%%%%%%%%%%%%%%%%%%%%%%%%%%

%=============================================================================

\bigskip
\setcounter{section}{-1}
\section{Introduction}

%-----------------------------------------------------------------------------

\noindent
In a recent series of works the first two
authors studied certain functors between categories of modules
of the complex general linear Lie algebra $\gl_m$ and of the
Yangian $\YY\,$; see \cite{KN1,KN2} and references therein.
Using the Howe duality \cite{H1,H2} between the Lie algebras
$\gl_m$ and $\gl_n\,$,
these functors arise from the centralizer construction
of the Yangian $\YY$ due to Olshanski \cite{O}.
They can also be defined as direct sums over $N=1,2,\ldots$
of the compositions of two well known functors.
The first functor in the composition, due to Cherednik \cite{C2},
is between the categories of modules
of $\gl_m$ and of the degenerate affine Hecke algebra
of $\gl_N\,$. The second functor in the composition, due to
Drinfeld \cite{D1},
is between the latter category and the category of
$\YY$-modules.

In the above mentioned series of works
the Zhelobenko operators \cite{Z} were used to study
intertwining operators between certain
%\emph{standard\/}
$\YY$-modules.
These modules correspond to pairs of weights
$\la=(\la_1,\ldots,\la_m)$ and $\mu=(\mu_1,\ldots,\mu_m)$
of the Lie algebra $\gl_m$
%(elements of the dual space $\h^*$ to Cartan subalgebra of Lie
%algebra $\gl_m$)
subject to a condition that each difference
$\nu_a=\lambda_a-\mu_a$
is a non-negative integer not exceeding $n\,$.
In the present article we denote
by $\Mlm$ the corresponding $\YY$-module, see \eqref{starat}.
By definition, this module is a tensor product over
$a=1,\ldots,m$
of certain pullbacks of the fundamental modules
$\Lambda^{\nu_a}(\CC^n)$ of $\gl_n\,$.
Each pullback also depends on a complex number $\mu_a\,$,
while the tensor product is taken by using the comultiplication
\eqref{Delta} on $\YY\,$.
In particular, when $\nu_a=1$ this pullback is called
a \text{vector\/} $\YY$-module.
Note that in the present article we use notation different from that
of \cite{KN1,KN2}. Most significantly, here the weights $\la$ and $\mu$ 
of the $\Mlm$ correspond to the weights $\la+\rho$ and $\mu+\rho$ in
\cite{KN1,KN2} where $\rho=(0,-1\sco 1-m)\,$.

Regard the symmetric group $\S_m$ as the Weyl group of $\gl_m$
and let $\si_0\in\S_m$ be the longest element.
One of the principal results of \cite{KN2}
was a new proof of irreducibility of the image
of the canonical intertwining operator
\begin{equation}
\label{ilamu}
\Mlm\to\Phi_{\,\si_0(\mu)}^{\,\si_0(\lambda)}
\end{equation}
under the condition of \emph{dominance\/}
of the weight $\la$ of $\gl_m\,$.
This condition on $\la$ means that
\begin{equation}
\label{domcon}
\la_a-\la_b\neq-1,-2,\ldots
\quad\text{whenever}\quad
a<b\,.
\end{equation}
The study of these %intertwining
operators has been started by Cherednik \cite{C1}.
The proof of irreducibility in \cite{KN2}
was based on the results of \cite{KNV}.
Other proofs follow from the results of
Akasaka and Kashiwara~\cite{AK} and of Nazarov and Tarasov \cite{NT}.
Note that a connection between the intertwining operators for 
tensor products of $\YY$-modules
and the Zhelobenko operators has been first discovered 
by Tarasov and Varchenko \cite{TV}.

It was also demonstrated in \cite{KN2}
that up to an automorphism of the form 
\eqref{similar} of the algebra $\YY\,$,
any irreducible finite-dimensional
$\YY$-module arises as the image of some
intertwining operator \eqref{ilamu}.
Furthermore, for that particular purpose it suffices to
consider only the operators corresponding
to dominant weights $\la$ while $\mu$ satisfy the extra conditions
\begin{equation}
\label{mucon}
\mu_a-\mu_b\le0
\quad\text{whenever}\quad
\la_a=\la_b
\quad\text{and}\quad
a<b\,.
\end{equation}
For any dominant $\la$
the canonical intertwining operator \eqref{ilamu}
is defined only up to a scalar multiplier.
These multipliers can be chosen in such a way that
for any fixed $\nu=(\nu_1,\ldots,\nu_m)$
the operator \eqref{ilamu} depends rationally 
on the weight $\la\ts$. One such choice was made in \cite{KN2}.

In the normalization used in \cite{KN2}
the operator \eqref{ilamu} vanishes if
any of the extra conditions \eqref{mucon} is not satisfied.
However, other arguments \cite{AK,KN2,R,Zel} indicate that
in this case the $\YY$-module $\Mlm$
should still have a unique irreducible quotient.
Another normalization of \eqref{ilamu}
was then provided in \cite{KNP}. In that normalization
for any dominant $\la$ the operator \eqref{ilamu} does not vanish, 
and moreover is constructed explicitly
by using the \emph{fusion procedure\/} from \cite{C1}.
As a $\YY$-module the image of %the operator
\eqref{ilamu} is then isomorphic to
that of the operator obtained by replacing in \eqref{ilamu}
the weight $\mu$ by any weight $\si(\mu)$ where
$\si\in\S_m$ fixes %the weight
$\la\,$.

The construction in \cite{KNP} is combinatorial
and rather involved. The first goal of the present article
is to provide another explicit expression \eqref{i}
for the canonical intertwining operator \eqref{ilamu}
which is easier to use.
In this form the operator \eqref{ilamu}
has the same normalization as in \cite{KNP}.
%also does not vanish for any dominant $\la\,$.
Yet the existence of this new form of \eqref{ilamu}
is far from obvious. It is obtained by also using
another kind of Zhelobenko operators.
The two kinds of these operators were studied in \cite{KO}
and are closely related.
They are proportional for the generic values of their parameters
and admit analytic continuations to the
domains of non-singularity of each other.
By \cite{KN1} our %renormalized 
intertwining operator \eqref{ilamu}
corresponds to a product of
Zhelobenko~operators~of~both~kinds.

A more general goal of the present article is to
eliminate the use of the automorphisms \eqref{similar}
in \cite{KN2,KNP}. %as discussed above.
The first steps in this direction have been
made in \cite{N,S}. In particular,
the Howe duality between
the Lie algebras ${\mathfrak{u}}_{\,p,q}$ and $\gl_n$
has been used~in~\cite{S} to construct $\YY$-modules
which do not arise as subquotients of tensor products of
the vector modules. Thus we come to the notions of
\text{polynomial\/} and \text{rational\/} modules
naturally generalizing the corresponding
notions for the general linear group $\GL_n\,$:
the polynomial modules are subquotients
of tensor products of vector modules,
the rational modules are
subquotients of tensor products of vector and
\text{covector\/} modules. These covector modules can be described
as the pullbacks of the vector modules relative to the automorphism
\eqref{omega}~of~$\YY\,$.

Up to twisting by the automorphisms of $\YY$
of the form \eqref{similar}, the isomorphism classes
of irreducible finite-dimensional
$\YY$-modules are labelled by
sequences of $n-1$ monic polynomials 
with complex coefficients, 
called the \text{Drinfeld polynomials\/} \cite{D2}.
The polynomial and the rational irreducible
$\YY$-modules have parametrizations very similar to that~of~\cite{D2}. 
In addition to the $n-1$ Drinfeld polynomials 
we have one more monic polynomial for labelling the
polynomial irreducible finite-dimensional $\YY$-modules, 
or one more rational function for labelling
the rational irreducible finite-dimensional $\YY$-modules.
This result is stated as Theorem \ref{prop-hw} below.
It comes from the analysis of the action of the Yangian $\YY$
on the  tensor products of vector and covector modules.
Moreover, we demonstrate that all the polynomial and the rational 
irreducible $\YY$-modules arise as images of
the intertwining operators %\eqref{ilamu} 
studied in \cite{KN1,KN2} and in \cite{S} respectively.

Our article is organized as follows. Section 1 recalls
the basic definitions and results from
the representation theory of the Yangian $\YY\,$.
Sections 2 and 3 give realizations of the polynomial and the
rational $\YY$-modules. All proofs are gathered in Section 4.

%=============================================================================

\bigskip\section{Basic results}

%-----------------------------------------------------------------------------

\subsection{}

The \textit{Yangian\/} $\YY$ of the general linear Lie algebra $\gl_n$
is a complex unital associative algebra with a family of generators
$T_{ij}^{(1)}, \, T_{ij}^{(2)} \etc$ with $i,j=1,\dots,n\,$.
These generators are customarily
gathered into the generating series
$$
T_{ij}(u)=
\de_{ij}+T_{ij}^{\,(1)}u^{-1}+T_{ij}^{\,(2)}u^{-2}+\ldots
$$
where $u$ is a formal parameter. The defining relations
of $\YY$ can be then written as
$$
(u-v)\,[\,T_{ij}(u),T_{kl}(v)\,]=
T_{kj}(u)T_{il}(v)-T_{kj}(v)T_{il}(u)
$$
where $v$ is another formal parameter. These relations
imply that for any formal power series
$f(u)$ in $u^{-1}$ with the leading term 1 and all the coefficients in 
$\CC$ the assignments
\begin{equation}
\label{similar}
T_{ij}(u)\mapsto\, f(u)\,T_{ij}(u)
%\quad\textrm{for}\quad i,j=1,\dots,n
\end{equation}
define an automorphism of the algebra $\YY\,$.
Further, for any $z\in\Cbb$ the assignments
\beq
\label{tauz}
\tau_z\colon\,T_{ij}(u)\mapsto\,T_{ij}(u-z)
%\quad\textrm{for}\quad i,j=1,\dots,n
\eeq
define an automorphism $\tau_z$ of $\YY\,$.
Here each of the formal series $T_{ij}(u-z)$ in $(u-z)^{-1}$ 
should be re-expanded in $u^{-1}$ so that the assignment
\eqref{tauz} becomes a correspondence between
the respective coefficients of the series in $u^{-1}\,$.
In this article we also employ the involutive automorphism
$\om$ of $\YY$ defined by the assignments
\begin{equation}
\label{omega}
\om\colon\,T_{ij}(u)\mapsto T_{ji}(-u)\,.
%\quad\text{for}\quad i,j=1\sco n.
\end{equation}

The quotient {of the} algebra $\YY$ by the relations
$T_{ij}^{\,(2)}=T_{ij}^{\,(3)}=\ldots=0$ for all indices $i,j=1\sco n$
is isomorphic to the universal enveloping algebra
$\,\U(\gl_n)\,$. The defining relations of $\YY$
contain in particular the commutation relations
$$
[\,T_{ij}^{(1)},T_{kl}^{(1)}\,]=
\de_{kj}\,T_{il}^{(1)}-\de_{il}\,T_{kj}^{(1)}.
$$
Hence the images of the generators
$T_{ij}^{(1)}$ in the quotient algebra can be identified
with the standard matrix units in $\gl_n$ having the same
indices $i,j\,$. Denote by $\pi$
the corresponding quotient map $\YY\to\,\U(\gl_n)\,$.
We will also use the composite homomorphisms
$\ev_z\,,\evp_z:\,\YY\to\U(\gl_n)$ defined for any parameter $z\in\CC$ by
\begin{equation}
\label{evz}
\ev_z= \pi\,\comp\,\tau_z\,,\quad
\evp_z =\pi\,\comp\,\om\,\comp\,\tau_z\,.
\end{equation}

The associative algebra $\YY$ contains $\,\U(\gl_n)\,$
as a subalgebra. Again, here~we~identify the generators 
$T_{ij}^{(1)}$ with the corresponding matrix units in $\gl_n\,$.
Hence the homomorphism $\pi$ is identical on the subalgebra
$\,\U(\gl_n)\subset\YY\,$. The same is true for $\ev_z$ with any $z\in\CC\,$.

%Note that in \cite{S} \,$\ev_z = \pi\comp\tau_{-z}$
%and $\evp_z=\pi\comp\tau_z\comp\om$.}

The Yangian $\YY$ is a Hopf algebra.
The counit homomorphism $\YY\to\Cbb$ and the antipodal antihomomorphism 
$\YY\to\YY$ are defined by
mappings $T_{ij}(u)\mapsto \delta_{ij}$ and
$T(u)\mapsto T^{\,-1}(u)$ respectively, while
the comultiplication $\YY\to\YY$
is defined~by
\begin{equation}
\label{Delta}
T_{ij}(u)\,\mapsto\,\sum_{k=1}^n T_{ik}(u)\otimes T_{kj}(u)\,.
\end{equation}

%-----------------------------------------------------------------------------

\subsection{}

For any non-negative integer $d$
consider the exterior power $\La^d(\Cbb^n)$ of the defining 
$\gl_n$-module $\Cbb^n$.
Let us denote by $\Phi^{\hspace{.5pt}d}_z$ and $\Phi^{-d}_z$
the $\YY$-modules obtained by pulling
the $\gl_n$-module
$\La^d(\Cbb^n)$ back through the homomorphisms $\ev_z$ and $\evp_{z-1}$ 
respectively, see \eqref{evz}. {Clearly,
these $\YY$-modules are non-zero if and only if $d\le n\,$.

For any $z\in\CC$ the $\YY$-modules $\Phi^{1}_z$ and $\Phi^{-1}_z$ 
are called \textit{vector\/}
and \textit{covector\/} modules respectively}.
Their underlying space is $\CC^n\ts$.
Let $\hc{e_1\sco e_n}$ be the standard basis in $\Cbb^n$.
The action of $\YY$ on the basis vectors is given by
\begin{align}
\label{vectorrep}
T_{ij}(u)\,e_k
\,=\,
\delta_{ij}\,e_k+\frac{\delta_{jk}\,e_i}{u-z}
\hspace{20.5pt}
&\ \quad\text{in}\ \quad\Phi^{1}_z\,,
\\
\label{covectorrep}
T_{ij}(u)\,e_k
\,=\,
\delta_{ij}\,e_k-\frac{\delta_{ik}\,e_j}{u-z+1}
&\ \quad\text{in}\ \quad\Phi^{-1}_z\,.
\end{align}
%by the definitions \eqref{tauz} and \eqref{evz}.
Further, for any $z\in\CC$ put
\begin{equation}
\label{detmod}
\De_z=\Phi_z^n
\quad\text{and}\quad
\Dep_z=\Phi_z^{-n}\,.
\end{equation}
{We call $\De_z$ and $\Dep_z$} the \textit{determinantal\/}
$\YY$-modules. By definition,
an $\YY$-module is \textit{polynomial\/} if it is
isomorphic to a subquotient of a tensor product of the vector
modules. More generally, an
$\YY$-module is \textit{rational\/} if it is isomorphic
to a subquotient of a tensor product of vector and covector modules. 
%Using this terminology, we can state the following
In Subsection~\ref{sec:pf-mod} we prove the following

\begin{prop}
\label{prop-mod}
If\/ $d\in\{0,1\sco n\}$ and $z\in\CC$ then\,:

\smallskip

(i) the module $\Phi^d_z$ is polynomial\,;

(ii) the module $\Phi^{-d}_z$ is rational\,;

(iii) the one-dimensional module $\Phi_z^{\hspace{.5pt}0}$ is trivial\,; 

(iv) the one-dimensional modules $\De_z$ and $\Dep_z$ are cocentral\,;

(v) the module $\Phi^{-d}_z$ is isomorphic to\/
$\Phi^{\ts n-d}_z\otimes\Dep_z\,$.

\end{prop}

\begin{remark}
By the last two parts of Proposition \ref{prop-mod}  
any rational $\YY$-module is isomorphic to a tensor product of a
polynomial $\YY$-module and a number of determinantal modules 
of the form $\De'_z\,$.
Further, every polynomial or rational $\YY$-module
admits an action of the complex general linear group ${\rm GL\ts}_n$
compatible with the embedding $\U(\gl_n)\subset\YY$ as described above.
This action of the group $\operatorname{GL\ts}_n$ 
is polynomial or rational respectively.
\end{remark}

Now let $\la=(\la_1,\ldots,\la_m)$
and $\mu=(\mu_1,\ldots,\mu_m)$ be two
elements of $\CC^m$ such that each difference
$\nu_a=\lambda_a-\mu_a$
is an integer.
The corresponding
\textit{standard rational\/} $\YY$-module
is the tensor product
\begin{equation}
\label{starat}
\Mlm=
\Phi_{\mu_1}^{\ts\nu_1}
\otimes\ldots\otimes
\Phi_{\mu_m}^{\ts\nu_m}\,.
\end{equation}
We will be occasionally calling $\la$ and $\mu$ the weights of $\Mlm\,$.
If $\nu_1\sco\nu_m\ge0$ then $\Mlm$ is the 
\textit{standard polynomial\/} $\YY$-module as employed
in \eqref{ilamu}. Note that by Proposition \ref{prop-mod}
the standard polynomial $\YY$-modules
are indeed polynomial in our terminology, and
the standard rational $\YY$-modules
are indeed rational.

%-----------------------------------------------------------------------------

\subsection{}

\,\ts
Let us recall {some} basic facts {from} the
representation theory of the Yangian~$\YY\,$.
They were first obtained by Tarasov \cite{T1,T2} in the case
$n=2$ and then generalized by Drinfeld \cite{D2} to any $n\,$.
For a detailed exposition of the
proofs of these basic facts see \cite{M}.

Let $\Psi$ be any irreducible
finite-dimensional $\YY$-module.
There exists a non-zero vector
$v\in\Psi$ called \textit{highest\/},
such that $T_{ij}(u)\,v=0$ for all $i<j\,$.
This vector is unique up to a multiplier from $\CC\,$.
For every index $i$ there exists a formal power series $A_i(u)$
in $u^{-1}$ with coefficients in $\Cbb$ and the leading term $1$
such that $T_{ii}(u)\,v=A_i(u)\,v$.
For any $i<n$ we have
\begin{equation}
\label{drinpol}
\frac{A_i(u)}{A_{i+1}(u)}=\frac{P_i(u+1)}{P_i(u)}
\end{equation}
where $P_i(u)$ are monic polynomials in $u\,$.
The polynomials $P_1(u)\sco P_{n-1}(u)$ are called the
\textit{Drinfeld polynomials\/} of $\Psi\,$.
The series $A_1(u)\sco A_n(u)$ determine the
$\YY$-module $\Psi$ uniquely up to isomorphism.
%We will call series $A_i(u)$ the \emph{highest weight series}
%of the module $\Psi$.
Any sequence of formal power series $A_1(u)\sco A_n(u)$
in $u^{-1}$ with coefficients in $\CC$ and the leading
terms $1\,$ satisfying \eqref{drinpol}
for some monic polynomials $P_1(u)\sco P_{n-1}(u)$ occurs in this way.

The sequence of series $A_1(u)\sco A_n(u)$
satisfying \eqref{drinpol}
can be recovered from any one of them, say from the $A_n(u)\ts$,
and from the sequence of monic polynomials $P_1(u)\sco P_{n-1}(u)\,$.
Here the series $A_n(u)$ itself can be chosen arbitrary, 
provided its leading term 
is $1\,$. Our first (and rather elementary) result is the next theorem.
We prove it in Subsections~\ref{sec:pf-hw}~and~\ref{sec:pf-hw-if}.

\begin{theorem}
\label{prop-hw}
The irreducible finite-dimensional $\YY$-module
determined by $A_n(u)$ and by
$P_1(u)\sco P_{n-1}(u)$ is polynomial or rational respectively
if and only if
\begin{equation}
\label{pn}
A_n(u)=\dfrac{Q_n(u+1)}{Q_n(u)}
\end{equation}
for some polynomial or rational function $Q_n(u)$ in $u\,$.
\end{theorem}

Clearly, in the case of a polynomial $\YY$-module
the polynomial function $Q_n(u)$ in \eqref{pn}
can be chosen monic. More generally, the rational
function $Q_n(u)$ can be chosen as a ratio of two
monic polynomials. 
This choice defines the function $Q_n(u)$ in \eqref{pn} uniquely.

Note that equalities \eqref{drinpol} and \eqref{pn}
imply that for each $i=1\sco n$
\begin{equation}
\label{qi}
A_i(u)=\dfrac{Q_i(u+1)}{Q_i(u)}
\qquad\text{where}\qquad 
Q_i(u)=P_i(u)\ldots P_{n-1}(u)\,Q_n(u)\,.
\end{equation}
The functions $Q_1(u)\sco Q_n(u)$
are polynomial or rational if the corresponding
irreducible finite-dimensional $\YY$-module
is respectively polynomial or rational.
Here the converse (the only if) statement is untrue,
because by definition the functions $Q_1(u)\sco Q_n(u)$
are related by polynomials $P_1(u)\sco P_{n-1}(u)\,$.

%=============================================================================

\bigskip\section{Polynomial modules}

%-----------------------------------------------------------------------------

\subsection{}
\label{sec:2.1}

For any positive integer $m$ take the general linear Lie algebra $\gl_m\,$.
For $a,b=1\sco m$ let $E_{ab}\in\gl_m$ be the standard matrix units.
Let $\h$ be the Cartan subalgebra of $\gl_m$ with the basis
$\{E_{11}\sco E_{mm}\}\,$. Identify the dual vector space $\h^*$
with $\Cbb^m$ by using this basis~of~$\h\,$.
Hence any element of $\aw\in\h^*$ will be written as 
$(\aw_1\sco\aw_m)\in\Cbb^m\,$.
The Weyl group of $\gl_m$ is isomorphic to the symmetric group 
$\S_m$. Thus there is
a natural action of the group $\S_m$ on $\h^*$ 
such that $\sigma(\aw)_a=\aw_{\,\sigma^{-1}(a)}$ for $\sigma\in\S_m\,$.
We denote by $\sigma_0$ the longest element of $\S_m$
so that we have $\sigma_0(a)=m+1-a$ for each $a=1\sco m\,$.

Choose the triangular decomposition 
$\gl_m = \n\oplus\h\oplus\n'$ where $\n$ and $\n'$ are the opposite 
nilpotent subalgebras of $\gl_m$
spanned by the vectors $E_{ab}$ where $a>b$ and $a<b$ respectively.
Then the positive roots of $\gl_m$ are the 
weights $\la\in\h^*$ with the only non-zero coordinates $\la_a=1$ and  
$\la_b=-1$ for some pair of indices $a<b\,$.

Denote by $\G_{mn}$ the Grassmann algebra on the $m\ts n$ anticommuting
variables $x_{ai}$ where $a=1\sco m$ and $i=1\sco n\,$.
Let $\d_{ai}$ be the operator of left derivation in $\G_{mn}$
relative to $x_{ai}\,$. Then for any $a,b=1\sco m$ consider
the first order differential operator on $\G_{mn}$
\begin{equation*}
\label{phi}
\E_{ab}\,=\,\sum_{k=1}^n\,x_{ak}\,\d_{\ts bk}\,.
\end{equation*}
The assignment $E_{ab}\mapsto\E_{ab}$ defines an action
of the Lie algebra $\gl_m$ on $\G_{mn}\,$.
We will use the weight decomposition of $\G_{mn}$
under the action of $\h\subset\gl_m\,$. 
We will also use the  action of the symmetric group $\S_m$ on 
$\G_{mn}$ by permutations of the first indices of the variables:
$$
\sigma(x_{ai})=x_{\sigma(a)\ts i}
\quad\text{for any}\quad
\sigma\in\S_m\,.
$$

Now take any two indices $a$ and $b$ such that $a<b\,$.
Let $\aw\in\h^*$ be any weight such that
$\aw_a-\aw_b\neq-1,-2,\ldots\,\,$. 
Define the linear operators $X^\aw_{ab}$ and $Y^\aw_{ab}$
on the vector space $\G_{mn}$ by % setting
$$
X^\aw_{ab}\,=\,1+\sum_{r=1}^\infty\,
\frac{(-1)^r\E_{ba}^{\ts r}\,\E_{ab}^{\ts r}}
{r!\,(\aw_a-\aw_b+1)_{\ts r}}
\qquad\text{and}\qquad
Y^\aw_{ab}\,=\,1+\sum_{r=1}^\infty\,
\frac{(-1)^r\E_{ab}^{\ts r}\,\E_{ba}^{\ts r}}
{r!\,(\aw_a-\aw_b+1)_{\ts r}}
$$
where $(u)_{\ts r}=u\,(u+1)\ldots(u+r-1)$ is the Pochhammer function.
The operators $X^\aw_{ab}$ and $Y^\aw_{ab}$ on $\G_{mn}$
are well defined and preserve the subspace  
$\G^{\ts\nu}_{mn}\subset\G_{mn}$ of every weight
$\nu\in\h^*$.

%-----------------------------------------------------------------------------

\subsection{}
\label{sec:22}

For any two weights $\la,\mu\in\h^*$ let 
$\nu=\la-\mu$ be their difference.
Suppose that each coordinate $\nu_a$ of $\nu$ is a non-negative integer
not exceeding $n\,$.
Let us consider the standard polynomial $\YY$-module $\Mlm\,$.
By the definition \eqref{starat}
its underlying vector space is 
\begin{equation}
\label{underspace}
\La^{\nu_1}(\Cbb^n)\sot\La^{\nu_m}(\Cbb^n)\,.
\end{equation}
%Here in view of Proposition \ref{prop-mod}(iii)
%we are excluding the cases when any coordinate $\nu_a=0\,$.
Using the basis vectors $e_1\sco e_n\in\Cbb^n$
we can now define an isomorphism of vector~spaces
$$
\al_\nu:\,\Mlm\,\to\,\G_{mn}^{\ts\nu}
$$
as follows. For an element of the vector space \eqref{underspace}
$$
w\,=\,
(\ts e_{i_1}\sw e_{i_{\nu_1}})
\sot
(\ts e_{j_1}\sw e_{j_{\nu_m}})
$$
we set
\beq
\label{anu}
\al_\nu(w)\,=\,
(\ts x_{1\ts i_1}\ts\ldots\,x_{1\ts i_{\nu_1}})
\,\ldots\,
(\ts x_{m\ts j_1}\ts\ldots\,x_{m\ts j_{\nu_m}})\,.
\eeq

Now suppose that $\la\in\h^*$ satisfies 
the dominance condition \eqref{domcon}.
Define a linear map 
$$
I_\mu^\la\,:\Mlm\,\to\,\M_{\,\sigma_0(\mu)}^{\,\sigma_0(\la)}
$$
as the composition
\beq
\label{i}
I_\mu^\la\,=\,(-1)^N\,
\al_{\sigma_0(\nu)}^{\,-1}\, 
\sigma_0\,Z^\la_\mu\,\al_{\nu}
\eeq
where
\beq
\label{N}
N\ =\sum_{1\le a<b\le m}\nu_a\,\nu_b\,,
\eeq
\beq
\label{z}
Z^\la_\mu\ =
\prod_{1\le a<b\le m}^{\longrightarrow}
\left\{
\begin{array}{rl}
\,X_{ab}^\la&\text{if\ }\;\;\nu_a\ge\nu_b
\\[4pt]
\,Y_{ab}^\mu&\text{if\ }\;\;\nu_a<\nu_b
\end{array}
\right.
\eeq
and the ordering of the factors in the product
over the pairs of indices $a<b$ corresponds to any normal ordering\ts
%-----------------------------------------------------------------------------
\footnote{A total ordering $\prec$ 
of the system of positive roots
of a reductive Lie algebra is called \emph{normal} if for any positive roots 
$\alpha,\beta$ such that $\alpha\prec\beta$ and $\alpha+\beta$ 
is also a positive root, the relation 
$\alpha\prec (\alpha+\beta)\prec\beta$ holds\ts; see \cite{AST}. 
There is a natural bijection between the set of normal orderings
of the system of positive roots of a reductive Lie algebra and the 
set of reduced decompositions
of the longest element of the Weyl group \cite{Z}.}
%-----------------------------------------------------------------------------
of the positive roots of the Lie algebra $\gl_m\,$.
It means that the factor corresponding to the pair $a<c$
stands between those corresponding to $a<b$~and~to~$b<c\,$.
The right hand side of \rf{i} is defined due to the dominance of 
$\la\,$. In Subsections \ref{sec:pf-i-1} to \ref{sec:pf-i-3}
we will derive from the results of the works \cite{KN1,KN2,KNP}
the following proposition.

\begin{prop}
\label{prop-i} 
For any weights $\la,\mu\in\h^*$ such that $\la$ is dominant and all 
coordinates
$\nu_a=\lambda_a-\mu_a$ are non-negative integers not exceeding $n\ts$,
the operator $I_\mu^\la$ defined by \eqref{i}:

\smallskip

(i) is not zero\ts;

(ii) does not depend on the choice of the normal ordering\ts;

(iii) commutes with the action of the Yangian $\YY\,$;

(iv) has the image irreducible under the latter action. 
\end{prop}

%-----------------------------------------------------------------------------

\subsection{}

Let us describe the highest vector 
and the series $A_1(u)\sco A_n(u)$ for the irreducible $\YY$-module
\begin{equation}
\label{llamu}
\Psi^\la_\mu
\,=\,
\Img I_\mu^\la
\,\subset\,
\M_{\,\sigma_0(\mu)}^{\,\sigma_0(\la)}\,.
\end{equation}
Let 
\begin{equation}
\label{vd}
v_d=e_1\wedge\ldots\wedge e_d
\end{equation}
be the highest vector of
the $\gl_n$-module $\Lambda^d(\CC^n)\,$. Here we use the triangular
decomposition of the Lie algebra $\gl_n$ similar to that of $\gl_m\,$,
see Subsection \ref{sec:2.1}. For the vector
\begin{equation}
\label{vlm}
\vlm=v_{\nu_1}\sot v_{\nu_m}\in\Mlm
\end{equation}
by using \eqref{Delta} we get 
\beq
\label{tijvlm}
T_{ij}(u)\,\vlm=0
\quad\text{if}\quad
i<j\,.
\eeq
We also get
\beq
\label{tiivlm}
T_{ii}(u)\,\vlm=A_i(u)\,\vlm
\eeq
where for each index $i$
\begin{equation}
\label{hw}
A_i(u)\ =\prod_{a\,:\,\nu_a\ts\ge\ts i} \frac{u-\mu_a+1}{u-\mu_a}\ .
\end{equation}

In the course of proving our Proposition \ref{prop-i}
we will establish the equality
\begin{equation}
\label{ivlm}
I_\mu^\la\bigl(v_\mu^\la\ts\bigr)=v_{\ts\sigma_0(\mu)}^{\ts\sigma_0(\la)}
\end{equation}
where according to \eqref{vlm}
\begin{equation}
\label{vlms}
v_{\ts\sigma_0(\mu)}^{\ts\sigma_0(\la)}=v_{\nu_m}\sot v_{\nu_1}\,.
\end{equation}
Hence the irreducible $\YY$-module 
$\Psi_{\mu}^\la$ has the highest vector
$v_{\sigma_0(\mu)}^{\sigma_0(\la)}$
and the corresponding series $A_1(u)\sco A_n(u)$
are given by \eqref{hw}. By \eqref{drinpol},\eqref{pn},\eqref{qi} 
it now follows that~for~$\Psi_{\mu}^\la$ 
%the polynomials $P_1(u)\sco P_{n-1}(u)$ and $Q_1(u)\sco Q_n(u)$ are given by
\begin{align*}
P_i(u)\ 
&=\prods{a\,:\,\nu_a\ts=\ts i}(\ts u-\mu_a)
\ \quad\text{where}\ \quad
i=1\sco n-1\ts;
\\
Q_i(u)\ 
&=\prods{a\,:\,\nu_a\ts\ge\ts i}(\ts u-\mu_a)
\ \quad\text{where}\ \quad
i=1\sco n\ts.
\end{align*}
While proving
Theorem~\ref{prop-hw} we will also obtain the following corollary to
Proposition~\ref{prop-i}.

\begin{corollary}
\label{prop-irrep}
The\/ %modules 
$\Psi^\la_\mu$ exhaust all the 
non-trivial irreducible polynomial $\YY$-modules.
\end{corollary}

For any $\sigma\in\S_m$ such that both $\la$ and $\sigma(\la)$ are 
dominant, the proof of \cite[Proposition~2.9]{KNP} yields
an isomorphism 
$$
\Psi^\la_\mu\,\to\,\Psi_{\sigma(\mu)}^{\sigma(\la)}
$$
of $\YY$-modules. Further suppose that all coordinates of 
the weight $\nu$ are positive, see Proposition \ref{prop-mod}(iii).
Then take another irreducible polynomial $\YY$-module $\Psi_{\mup}^{\la'}$
where $\la'$ is dominant and all coordinates of the weight $\la'-\mup$ are
positive integers not exceeding~$n\,$.
The $\YY$-modules $\Psi_\mu^\lambda$ and 
$\Psi_{\mup}^{\la'}$ are isomorphic
if and only if $\la\,,\mu$ and $\la',\mup$ are weights of the {same} 
$\gl_m$ and there exists $\sigma\in\S_m$ such that
$\la'=\sigma(\la)$ and $\mup=\sigma(\mu)\,$. 

%=============================================================================

\smallskip\section{Rational modules}

%-----------------------------------------------------------------------------

\subsection{}
\label{sec:rat}

Take a sequence of signs
$\sign=(\sign_1,\ldots,\sign_m)$ where $\sign_a\in\{-1,+1\}$
for each $a=1\sco m\,$. For any given $a,b=1\sco m$
let $\EE_{ab}$ be a
differential operator on $\G_{mn}$ of order at most two,
\begin{equation*}
\label{e}
\EE_{ab}\,=\,\suml{i=1}{n}\,q_{ai}\,p_{\ts bi}
\end{equation*}
where
\beq
\label{qp}
q_{ai} =
\left\{
\begin{array}{rl}
x_{ai}&\text{if}\quad\sign_a=+1
\\
\partial_{ai} &\text{if}\quad\sign_a=-1
\end{array}
\right.
\qquad\text{and}\qquad
p_{bi} =
\left\{
\begin{array}{rl}
\partial_{bi} &\text{if}\quad\sign_b=+1
\\
x_{bi}&\text{if}\quad\sign_b=-1\,.
\end{array}
\right.
\eeq

Now take any two indices $a$ and $b$ such that $a<b\,$.
Let $\aw\in\h^*$ be any weight such that
$\aw_a-\aw_b\neq-1,-2,\ldots\,\,$. 
Define the linear operators 
$X^{\aw\ts\sign}_{ab}$ and $Y^{\aw\ts\sign}_{ab}$
on the vector space $\G_{mn}$ by %setting
$$
X^{\aw\ts\sign}_{ab}\,=\,1+\sum_{r=1}^\infty\,
\frac{(-1)^r(\EE_{ba})^{\ts r}\,(\EE_{ab})^{\ts r}}
{r!\,(\aw_a-\aw_b+1)_{\ts r}}
\qquad\text{and}\qquad
Y^{\aw\ts\sign}_{ab}\,=\,1+\sum_{r=1}^\infty\,
\frac{(-1)^r(\EE_{ab})^{\ts r}\,(\EE_{ba})^{\ts r}}
{r!\,(\aw_a-\aw_b+1)_{\ts r}}
$$
where %in the denominators 
we again use the Pochhammer function 
$(u)_{\ts r}=u\,(u+1)\ldots(u+r-1)\,$.

Take any two weights $\la,\mu\in\h^*$ such that 
each coordinate $\nu_a$ of the difference
$\nu=\la-\mu$ is an integer with the absolute value
$|\nu_a|\le n\,$. Consider the corresponding
standard rational $\YY$-module $\Mlm\,$.
By the definition \eqref{starat}
its underlying vector space is 
\begin{equation}
\label{ratspace}
\La^{|\nu_1|}(\Cbb^n)\sot\La^{|\nu_m|}(\Cbb^n)\,.
\end{equation}
Denote by $|\nu|$ the weight 
$(|\nu_1|\sco|\nu_m|)$ of $\gl_m\,$.
We will use the linear map $\al_{|\nu|}$
to identify the vector space \eqref{ratspace}
with the weight subspace 
$\G_{mn}^{\ts|\nu|}\subset\G_{mn}\,$, see the definition \rf{anu}.

Choose the sequence
$\sign = \hr{\sign_1, \sign_2 \sco \sign_m}$ 
so that $\sign_a$ is $+1$ or $-1$ depending on whether
$\nu_a\ge0$ or $\nu_a<0\ts$. %for $a=1\sco m\,$.
Define the weights $\lap$ and $\nup$ of $\gl_m$
by setting $\lap=\mu+\nup$ where
$$
\nup_a=
\left\{
\begin{array}{cl}
\nu_a&\text{if}\quad\nu_a\ge0 
\\
n+\nu_a&\text{if}\quad\nu_a<0\,.
\end{array}
\right.
$$
Then denote
\beq
\label{Nb}
\Nb\ =\sum_{1\le a<b\le m}\nup_a\,\nup_b\,.
\eeq

Suppose that $\lap$ is dominant. Note that here we
do not require the dominance of $\la$~itself. Define a linear map 
$$
I_\mu^\la\,:\Mlm\,\to\,\M_{\,\sigma_0(\mu)}^{\,\sigma_0(\la)}
$$
as the composition
\beq
\label{ri}
I_\mu^\la\,=\,(-1)^{\Nb}\,
\al_{\sigma_0\ts|\nu|}^{\,-1}\, 
\sigma_0\,Z^\la_\mu\, 
\al_{|\nu|}
\eeq
where now
$$
Z^\la_\mu\ =
\prod_{1\le a<b\le m}^{\longrightarrow}
\left\{
\begin{array}{cl}
\,X_{ab}^{\lap\ts\sign}&\text{if\ }\;\;\nup_a\ge\nup_b
\\[4pt]
\,Y_{ab}^{\mu\ts\sign}&\text{if\ }\;\;\nup_a<\nup_b
\end{array}
\right.
$$
and the ordering of the factors in the product
over the pairs of indices $a<b$ corresponds to any normal ordering
of the positive roots of the Lie algebra $\gl_m\,$.
The next theorem extends our Proposition \ref{prop-i} %from polynomial 
to rational $\YY$-modules. 
The proof will be given in Subsection~\ref{sec:pf-ri}. 

\begin{theorem}
\label{prop-ri} 
For any $\la,\mu\in\h^*$ such that $\lap$ is dominant and all 
coordinates of $\la-\mu$ are integers with 
absolute values not exceeding $n\,$,
the operator $I_\mu^\la$ defined by \eqref{ri}\,:

\smallskip

(i) is not zero\ts;

(ii) does not depend on the choice of the normal ordering\ts;

(iii) commutes with the action of the Yangian $\YY\,$;

(iv) has the image irreducible under the latter action. 
\end{theorem}

%-----------------------------------------------------------------------------

\subsection{}

Let us describe the highest vector and the series $A_1(u)\sco A_n(u)$
for the irreducible $\YY$-module defined as the image of the
operator \eqref{ri}. We will still denote this module
by $\Psi^\la_\mu$ thus generalizing the notation \eqref{llamu}.

For any positive integer $d\le n$
consider the $\gl_n$-module obtained by pulling
$\Lambda^d(\CC^n)$ back through the
restriction of the automorphism \eqref{omega}
to the subalgebra $\U(\gl_n)\subset\YY\,$. Let
$$
v_{-d}=e_{n-d+1}\wedge\ldots\wedge e_n
$$
be the highest vector of the resulting $\gl_n$-module.
Recall that the highest vector of $\Lambda^d(\CC^n)$ itself is 
\eqref{vd}. Define the vector $\vlm\in\Mlm$
as in \eqref{vlm} where every coordinate
$\nu_a$ can also be negative now. Then %the equalities 
\eqref{tijvlm} and \eqref{tiivlm} still hold,
however now for each $i=1\sco n$
\begin{equation}
\label{rhw}
A_i(u)\ =
\prods{a\,:\,\nu_a\ts<\ts i-n}\frac{u-\mu_a}{u-\mu_a+1}
\ \ \cdot
\prods{a\,:\,\nu_a\ts\ge\ts i}\frac{u-\mu_a+1}{u-\mu_a}\ .
\end{equation}

In the course of proving our Theorem \ref{prop-i}
we will show that the equalities \eqref{ivlm},\eqref{vlms}
also hold now. Hence 
$v_{\sigma_0(\mu)}^{\sigma_0(\la)}$ 
is the highest vector of $\Psi_{\mu}^\la$
and the corresponding series $A_1(u)\sco A_n(u)$
are given by \eqref{rhw}. %By \eqref{drinpol},\eqref{pn},\eqref{qi} 
It now follows that for $\Psi_{\mu}^\la$ 
\begin{align*}
P_i(u)\ 
&=\prods{a\,:\,\nu_a=i\ts,\ts i-n}(\ts u-\mu_a)
\ \quad\text{where}\ \quad
i=1\sco n-1\ts;
\\
Q_i(u)\ 
&=\prods{a\,:\,\nu_a\ts<\ts i-n}\frac1{u-\mu_a}
\ \,\cdot
\prods{a\,:\,\nu_a\ge i}(\ts u-\mu_a)
\ \quad\text{where}\ \quad
i=1\sco n\ts.
\end{align*}
While proving
Theorem~\ref{prop-hw} we will also obtain the following corollary to
Theorem~\ref{prop-ri}.

\begin{corollary}
\label{prop-rirrep}
The $\Psi^\la_\mu$ exhaust all the non-trivial
irreducible rational $\YY$-modules.
\end{corollary}

There is no uniqueness 
in realizing a given non-trivial irreducible
rational $\YY$-module as some $\Psi^\la_\mu$ where
$\lap$ is dominant and all coordinates of 
the weight $\la-\mu$ are non-zero integers with absolute values
not exceeding~$n\ts$. Take Proposition~\ref{prop-mod}(v) 
with $0<d<n$~as~an~example:
\begin{align*}
\Phi^{-d}_z&=
\Psi^\la_\mu
\quad\text{where}\quad
m=1\ts,\quad
\la_1=z-d
\quad\text{and}\quad
\mu_1=z\ts;
\\[2pt]
\Phi^{\ts n-d}_z\otimes\Dep_z&=
%\Phi^{\ts n-d}_z\otimes\Phi^{\ts-n}_z=
\Psi^\la_\mu
\quad\text{where}\quad
m=2\ts,\quad
(\la_1,\la_2)=(z+n-d,z-n)
\quad\text{and}\quad
(\mu_1,\mu_2)=(z,z)\ts.
\end{align*}
All choices of $\la$ and $\mu$ with minimal number of coordinates 
$m$ are described in %the end of 
Subsection~\ref{sec:pf-hw-if}. 

%=============================================================================

\bigskip\section{The proofs}

%-----------------------------------------------------------------------------

\subsection{}
\label{sec:pf-mod}

Here we prove Proposition~\ref{prop-mod}.
Part (i) is well known, see for instance
\cite[Lemma~2.1]{KNP}. Part (ii) immediately follows from (i)
because $\om\,\comp\,\tau_z=\tau_{-z}\,\comp\,\om\,$.
By definition, the trivial $\YY$-module is one-dimensional
and is defined by the counit homomorphism $\YY\to\CC\,$.
Composition of the homomorphism $\pi_z:\YY\to\U(\gl_n)$
with the representation of $\U(\gl_n)$ in $\La^0(\Cbb^n)$
coincides with the counit. Therefore we get (iii). Further,
the cocentrality in (iv) means that
for any $\YY$-module $\rm M$ the flip
of two tensor factors yields isomorphisms  
$$
{\rm M}\otimes\De_z\,\to\,\De_z\otimes{\rm M}
\quad\text{and}\quad
{\rm M}\otimes\Dep_z\,\to\,\Dep_z\otimes{\rm M}\,.
$$
Thus (iv) follows from the definition \eqref{Delta}
of the comultiplication and an observation that
the defining homomorphisms of one-dimensional
$\YY$-modules \eqref{detmod} are given respectively~by
$$
T_{ij}(u)\,\mapsto\,\de_{ij}\,\frac{u-z+1}{u-z}
\ \quad\text{and}\ \quad
T_{ij}(u)\,\mapsto\,\de_{ij}\,\frac{u-z}{u-z+1}\,\,.
$$
Using the latter description of the $\YY$-module $\Dep_z\,$,
the (v) follows from \cite[Lemma 2.3]{KNP}.
However, we will still include another proof of (v) here. 
It will be then used in Subsection~\ref{sec:pf-ri}.

Let $\G_n$ be the Grassmann algebra on $n$ variables $x_1 \sco x_n\,$.
For each $d=0,1\sco n$ let $\G_n^{\ts d}\subs\G_n$ be the subspace
of degree $d$. Define a bijective linear map
$\al_d:\La^d(\Cbb^n)\to\G_n^{\ts d}$ by %setting
$$
\al_d\,(e_{i_1}\sw e_{i_d})=x_{i_1}\ldots\ts x_{i_d}\,.
$$
Let us carry the structures of $\YY$-modules from 
$\Phi^{-d}_z$ and $\Phi^{\ts n-d}_z\otimes\Dep_z$
to the vector spaces $\G_n^{\ts d}$ and $\G_n^{\ts n-d}\otimes\G_n^{\ts n}$
via the maps $\al_d$ and $\al_{n-d}\otimes\al_n$ respectively.
The action of the Yangian $\YY$ on the resulting two modules can be
described by the assignments
$$
T_{ij}(u)
\mapsto
\de_{ij}-\frac{x_j\ts\d_i}{u-z+1}
\quad\text{and}\quad
T_{ij}(u)
\mapsto
\Bigl(\ts\de_{ij}+\frac{x_i\ts\d_j}{u-z}\,\Bigr)
\otimes
\frac{u-z}{u-z+1}
$$
respectively, 
where $\d_i$ is the operator of left derivation in $\G_n$ relative to $x_i\,$.

Let $\D_n$ be the ring of left 
differential operators on the Grassmann algebra~$\G_n\,$.
Denote by $S$ the involutive automorphism of $\D_n$ defined by setting
$S(x_i)=\d_i$ for $i=1\sco n\,$. Put
$$
R(x)=S(x)\cdot x_1\ldots x_n
\quad\text{for}\quad x\in\G_n\,.
$$
Note that the operator 
$R^{\ts2}$ is $(-1)^{n\ts(n-1)/2}$ times the identity map.
Then we have the relation
\beq
\label{rs}
R\ts(\ts Y(x))=S(Y)(R(x))
\quad\text{for}\quad Y\in\D_n\,.
\eeq
Indeed, if $Y$ is the operator of multiplication
by any $y\in\G_n$ then we get \eqref{rs} because
$S$ is a homomorphism:
$$
R\ts(y\ts x)=S(y\ts x)\cdot x_1\ldots x_n=S(y)(R(x))\,.
$$
By substituting here $R^{\ts-1}(x)$ for $x\ts$, switching 
the left and the right hand sides of the resulting equality,
and then applying $R$ to both sides we get 
$$
R(\ts S(y)(x))=y\ts R(x)=S(S(y))(R(x))\,.
$$
Thus we also get \eqref{rs} for $Y=S(y)$. Since the operators
$Y=y$ and $Y=S(y)$ generate the ring $\D_n$ we obtain 
the relation \eqref{rs} in general.

An isomorphism of $\YY$-modules 
$\G_n^{\ts d}\to\G_n^{\ts n-d}\otimes\G_n^{\ts n}$
can now be defined by mapping any element $x\in\G_n^{\ts d}$ to the element 
$R(x)\otimes(\ts x_1\ldots x_n)\,$.
Due to the relation \eqref{rs}
the intertwining property of this map follows from the operator equality 
$$
\de_{ij}-\frac{x_j\ts\d_i}{u-z+1}
\,=\,
\frac{u-z}{u-z+1}\,
\Bigl(\ts\de_{ij}+\frac{\d_i\ts x_j}{u-z}\ts\Bigr)
\,.
$$
Hence the $\YY$-modules $\G_n^{\ts d}$ and 
$G_n^{\ts n-d}\otimes\G_n^{\ts n}$ are isomorphic.
By the definition of these two modules,
we now get the part (v) of Proposition~\ref{prop-mod}.

%-----------------------------------------------------------------------------

\subsection{}
\label{sec:pf-hw}

In this subsection we prove the \emph{only if}
part of Theorem~\ref{prop-hw}. This part follows from 
a more general property of $\YY$-modules stated as the lemma below. Let 
$$
\M=\Phi^{\ts\sign_1}_{z_1}\otimes\ldots\otimes\Phi^{\ts\sign_m}_{z_m}
$$ 
be the tensor product of any number $m$ of the vector and covector
$\YY$-modules. Here we have $\sign_a\in\{-1,1\}$ and $z_a\in\CC$
for each index $a=1\sco m\,$.

\begin{lemma} 
\label{lemma}
For any index $i$ all eigenvalues of\/ $T_{ii}(u)$ on
the $\YY$-module\/ $\M$~have~the~form
$$
\prods{a\,\in\ts I}\,\frac{u-z_a+1}{u-z_a} 
\ \cdot\  
\prods{a\,\in\ts J}\,\frac{u-z_a}{u-z_a+1}
$$
where $I$ and $J$ are subsets of the sets
of indices $a$ such that\/ 
$\sign_a=1$ and\/ $\sign_a=-\ts1$ respectively.
\end{lemma}

\noindent{\it Proof.}\quad
We will prove the lemma by induction on $m\,$.
In the base case $m=0$ the statement of the lemma is trivial.
Suppose that $m>0$ and the statement is true for $m-1$ instead of $m\,$.
We will consider the cases $\sign_m=1$ and $\sign_m=-\ts1$ separately.
In both cases we will denote
$$
\M^{\,\prime}=
\Phi^{\ts\sign_1}_{z_1}\otimes\ldots\otimes\Phi^{\ts\sign_{m-1}}_{z_{m-1}}\,.
$$

First suppose $\sign_m=1\,$. 
For any index $i=1\sco n$ let $W_i\subset\CC^n$ 
be the span of the vectors 
$e_1\sco e_{i-1}, e_{i+1} \sco e_n\,$.
Then by \rf{Delta},\rf{vectorrep}
for any vector $w\in\M^{\,\prime}$ and for~any~index~$k\neq i$
\begin{gather*}
T_{ii}(u)\,(w\otimes e_k)=(\,T_{ii}(u)\,w)\otimes e_k\,,
\\[4pt]
T_{ii}(u)\,(w\otimes e_i)=\frac{u-z_m+1}{u-z_m}\ 
(\,T_{ii}(u)\,w)\otimes e_i 
\quad\text{mod}\quad 
\M^{\,\prime}\otimes W_i\,[[\ts u^{-1}\ts]]\,.
\end{gather*}
In particular, the action of the coefficients of the series $T_{ii}(u)$
on $\M$ preserves the subspace $\M^{\,\prime}\otimes W_i\,$.
Hence in this case 
any eigenvalue of $T_{ii}(u)$ on $\M$ is equal to 
an eigenvalue of $T_{ii}(u)$ on $\M^{\,\prime}$
multiplied either by 1 or by 
$$
\frac{u-z_m+1}{u-z_m}\,.
$$

Next suppose that $\sign_m=-\ts1\,$. 
Then by \rf{Delta},\rf{covectorrep}
for any $w\in\M^{\,\prime}$ and for any $k\neq i$
\begin{gather*}
T_{ii}(u)\,(w\otimes e_k)=(\,T_{ii}(u)\,w)\otimes e_k
\quad\text{mod}\quad 
\M^{\,\prime}\otimes e_i\,[[\ts u^{-1}\ts]]\,,
\\[4pt]
T_{ii}(u)\,(w\otimes e_i)= 
\frac{u-z_m}{u-z_m+1}\ (\,T_{ii}(u)\,w)\otimes e_i\,.
\end{gather*}
In particular, the action of the coefficients of the series $T_{ii}(u)$
on $\M$ preserves the subspace $\M^{\,\prime}\otimes e_i\,$.
Hence in this case 
any eigenvalue of $T_{ii}(u)$ on $\M$ is equal to 
an eigenvalue of $T_{ii}(u)$ on $\M^{\,\prime}$
multiplied either by 1 or by 
$$
\frac{u-z_m}{u-z_m+1}\,.
\eqno{\square}
$$

%-----------------------------------------------------------------------------

\subsection{}
\label{sec:pf-hw-if}

In this subsection we will prove the \emph{if} part of
Theorem~\ref{prop-hw}. Then we will derive our
Corollaries \ref{prop-irrep} and \ref{prop-rirrep}.
Let $P_1(u)\sco P_n(u)$ and $P_{-n}(u)$
be any monic polynomials in $u$ with complex coefficients.
Let $Q_n(u)=P_n(u)/P_{-n}(u)\,$. Assume that the polynomials
$P_n(u)$ and $P_{-n}(u)$ have no common zeroes. 
Determine the series $A_1(u)\sco A_n(u)$ by \eqref{qi}.

We need to prove that the irreducible finite-dimensional $\YY$-module
corresponding to $A_1(u)\sco A_n(u)$ is rational, 
and is polynomial~if~$P_{-n}(u)=1\ts$. Write
$$
P_i(u)=\prodl{s=1}{m_i}\,(\ts u-z_{\ts is}\ts)
\qquad\text{for}\qquad i=1\sco n,-n\,.
\hspace{-20pt}
$$
Set $m=m_1+\ldots+m_n+m_{-n}\,$. Denote by $\Pc$ the set of $m$ pairs
$(i,z_{\ts is})$ where $s=1\sco m_i$ and $i=1\sco n,-n\,$.
Let $\la$ and $\mu$ be any two weights of $\gl_m$ such that 
for $\nu=\la-\mu$ the set
$$
\{\,(\ts\nu_a,\mu_a\ts)\,|\,a=1\sco m\,\}=\Pc\,.
$$
Define the weight $\lap$ as in Subsection \ref{sec:rat}. 
Note that if $\nu_a<0$ for any $a$ here, then $\nu_a=-\ts n\,$.
In particular, if $P_{-n}(u)=1$ then all the coordinates 
$\nu_a$ are positive integers not exceeding~$n\ts$.

Consider the corresponding $\YY$-module $\Mlm$ and its vector
$\vlm\,$. Note that the definition of this vector
does not require the dominance of the weight $\lap\,$.
Moreover, the equalities \eqref{tijvlm} and
\eqref{tiivlm} hold for any $\lap\,$, not necessarily dominant.
It now follows that the $\YY$-module $\Mlm$ has an irreducible 
subquotient such that its highest vector is the image of $\vlm\,$,
see for instance the proof of \cite[Theorem 2.16]{CP}.
The series $A_1(u)\sco A_n(u)$ corresponding to this 
irreducible subquotient are given by \eqref{qi}.
Hence we get the \emph{if} part of Theorem~\ref{prop-hw}. 

So far the weights $\la$ and $\mu$ have been determined
up to any simultaneous permutation of their coordinates.
We can now choose $\la$ and $\mu$ so that the weight $\lap$ is dominant.
Then due to \eqref{ivlm}
the irreducible subquotient of $\Mlm$ considered above
becomes the quotient relative to the kernel of the operator
$I_\mu^\la\,$.   
This quotient is isomorphic to $\Psi^\la_\mu\,$.
Corollaries~\ref{prop-irrep}~and~\ref{prop-rirrep} now
follow from Proposition \ref{prop-i} and Theorem \ref{prop-ri} respectively,
by the \emph{only if} part of 
Theorem~\ref{prop-hw}.

Note that if $P_{-n}(u)\neq1\ts$, then the irreducible 
module $\Psi^\la_\mu$
considered above does not necessarily have the minimal
number $m$ possible for the given polynomials $P_1(u)\sco P_{n-1}(u)$ and
the rational function $Q_n(u)\ts$. Suppose the set $\Pc$ 
%defined by \eqref{zis} 
contains the two pairs
$(d,z)$ and $(-n,z)$ for some $z\in\CC$ and $d>0\,$. 
Then $d<n$ because the polynomials
$P_n(u)$ and $P_{-n}(u)$ have no common zeroes.
Replace the two pairs $(d,z)$ and $(-n,z)$
in $\Pc$ by the single pair $(d-n,z)\,$. 
Let $\Pc^{\ts\prime}$ be any set obtained by
repeating this replacement step until possible. 
This $\Pc^{\ts\prime}$ may be not unique.
However, all resulting sets $\Pc^{\ts\prime}$
have the same size which will denote by $m^{\ts\prime}\ts$.
Let $\la'$ and $\mup$ be any weights of $\gl_{m'}$ such that 
for~$\nu^{\,\prime}=\la'-\mup$
$$
\{\,(\ts\nu^{\,\prime}_a\ts,\mup_a\ts)
\,|\,
a=1\sco m^{\ts\prime}\,\}=\Pc^{\ts\prime}\,.
$$
Assuming that $\overline{\la'}$ is dominant,
the irreducible $\YY$-module $\Psi_{\mup}^{\la'}$ 
has the same polynomials
$P_1(u)\sco P_{n-1}(u)$ and rational function $Q_n(u)$ as %our initial 
$\Psi^\la_\mu\,$.
In particular, it is isomorphic to $\Psi^\la_\mu\,$.

%-----------------------------------------------------------------------------

\subsection{}
\label{sec:pf-i-1}

This and the next three subsections are devoted to the proof of 
Proposition~\ref{prop-i}. First we recall
the construction of Zhelobenko operators for
the Lie algebra $\gl_m$ from \cite{KO}\ts; see also \cite{KN1,KNV}.
Let $\A$ be any associative algebra containing $\U(\gl_m)$
as a subalgebra. Then $\A$ can be regarded as a $\U(\gl_m)$-bimodule.
Suppose that the following two conditions are satisfied:
\begin{itemize}
\item[(i)] the corresponding adjoint action of the Lie algebra 
$\gl_m$ on $\A$ is locally finite and can be lifted 
to an algebraic action of the group $\GL_m$ by automorphisms 
of the algebra $\A\,$;
\\[-24pt]
\item[(ii)] 
the algebra $\A$ contains a vector subspace $\V$ 
invariant under the action of $\GL_m\,$, 
such that multiplication in $\A$ gives an isomorphism of $\GL_m$-modules
$
\U(\gl_m)\otimes\V\to\A\,.
$
\end{itemize}

\vspace{-4pt}
Consider $\S_m$ as the subgroup
of $\GL_m$ consisting of the permutation matrices. 
Note that under the above conditions the group $\S_m$
acts by automorphisms on the algebra $\A\,$.
In the beginning of Subsection \ref{sec:2.1}
we fixed a triangular decomposition $\gl_m=\n\oplus\h\oplus\n'\,$.
Now let $\Uh$ and $\Ab$ be the rings of fractions of 
$\U(\h)$ and of $\A$ relative to the set of denominators,
multiplicatively generated by all the elements
$$
E_{aa}-E_{bb}+z\in\U(\h)
\quad\text{where}\quad
a<b
\quad\text{and}\quad
z\in\ZZ\,.
$$
For every index $a=1\sco m-1$ consider a
standard $\mathfrak{sl}_2$-triple in $\gl_m\,$:
\begin{equation}
\label{sltwo}
E_a=E_{a,a+1}\,,\quad
F_a=E_{a+1,a}
\quad\text{and}\quad
H_a=E_{aa}-E_{a+1,a+1}\,.
\end{equation}
Then one can define a linear map $\xic_a:\A\to\Ab/\n\Ab$
by setting for each $x\in\A$
$$
\xic_a(x)=
\sigma_{a}(x)+\sum_{r=1}^\infty\,
(r\ts!\,H_a(H_a-1)\ldots(H_a-r+1))^{-1}
E_a^{\,r}\ad_{F_a}^{\,r}(\sigma_{a}(x))
\ \operatorname{mod}\ \n\Ab
$$
where $\sigma_a\in\S_m$ 
is the transposition of $a$ and $a+1$, 
regarded as an element of $\Aut(\A)$.

The so defined map $\xic_a$ can be canonically 
extended to the \emph{Zhelobenko operators}
$$ 
\xic_a\colon\Ab/\n\Ab\to\Ab/\n\Ab
\qquad\text{and}\qquad
\xic_a\colon\Ab/(\n\Ab+\Ab\n'\ts)\to
\Ab/(\n\Ab+\Ab\n'\ts)
$$
which we denote by the same symbol 
$\xic_a$ with a slight abuse of notation.
By \cite{Z} the operators $\xic_1\sco\xic_{m-1}$
on $\Ab/\n\Ab$ and on $\Ab/(\n\Ab+\Ab\n'\ts)$
satisfy the braid group relations
\begin{align}
\label{braid3}
\xic_a\,\xic_{a+1}\,\xic_a 
&\,=\, 
\xic_{a+1}\,\xic_a\,\xic_{a+1} 
&&\hspace{-85pt}\text{for}\qquad a<m-1\,, 
\\
\label{braid2}
\xic_a\,\xic_b 
&\,=\, 
\xic_b\,\xic_a 
&&\hspace{-85pt}\text{for}\qquad |\,a-b\,|>1\,.
\end{align}
Moreover, for any $H\in\h$ and $x\in\Ab$ we have
$$
\xic_a\ts(H\ts x)=(\si_a\circ H)\,\xic_a(x)
\qquad\text{and}\qquad
\xic_a\ts(x\ts H)=\xic_a(x)\,(\si_a\circ H)
$$
where
\beq
\label{shiftuh}
\sigma_a\circ E_{bb}=E_{\sigma_a(b),\sigma_a(b)}
-\delta_{ab}+\delta_{a+1,b}
\eeq
defines the \emph{shifted} action of the Weyl group $\S_m$ on $\U(\h)\,$.

Another family of Zhelobenko operators on 
the double coset space $\Ab/(\n\Ab+\Ab\n'\ts)$ was also studied in \cite{KO}.
For any $a=1\sco m-1$ we can define a linear map 
$\zec_a:\A\to\Ab\ts/\Ab\n'$~by
$$
\zec_a(x)=\si_a(x)+\sum_{r=1}^\infty\,
(-1)^r\ad_{E_a}^{\,r}(\si_{a}(x))\,F_a^{\,r}\,
(r\ts!\,H_a\,(H_a-1)\ldots (H_a-r+1))^{-1}
\ \operatorname{mod}\ \Ab\n'\,.
$$
It can be then extended to the operators
$$
\zec_a\colon \Ab\ts/\Ab\n' \to \Ab\ts/\Ab\n'
\qquad\text{and}\qquad
\zec_a\colon\Ab/(\n\Ab+\Ab\n'\ts)\to\Ab/(\n\Ab+\Ab\n'\ts)
$$
which we denote by the same symbol $\zec_a$ 
by an abuse of notation.
The operators $\zec_1\sco\zec_{m-1}$
satisfy the same braid group relations as %the operators 
$\xic_1\sco\xic_{m-1}\,$. 
Moreover, for any $H\in\h$ and $x\in\Ab$ %we have
$$
\zec_a\ts(H\ts x)=(\si_a\circ H)\,\zec_a(x)
\qquad\text{and}\qquad
\zec_a\ts(x\ts H)=\zec_a(x)\,(\si_a\circ H)\,.
$$

For any $a=1\sco m-1$ we will now give a more explicit description 
of the operators $\xic_a$ and $\zec_a$ on 
$\Ab/(\n\Ab+\Ab\n'\ts)$. 
Let $\g_a\subset\gl_m$ be Lie subalgebra 
spanned by the three elements \eqref{sltwo}.
Choose $j\in\{\ts0,\frac12\,,1,\ldots\ts\}$ and take any $x\in\V$
from an irreducible $\g_a$-submodule of
$\V$ of dimension $2j+1\,$. Here we use the restriction
of the adjoint action of $\gl_m$ to the subspace $\V\subset\A\,$. 
Suppose that $x$ is of the weight $2\ts h$ relative to the action $H_a\,$:
$$
[\,H_a,x\,]=2\ts h\ts x
\quad\text{where}\quad
h\in\{-j,-j+1,\ldots, j-1,j\ts\}\,.
$$
Then by \cite[Section 1.4]{KN2} the double coset
$\xic_a(x)\in\Ab\ts/(\n\Ab+\Ab\n')$ is that of the element~of~$\Ab$
\begin{align*}
&\prod_{i=1}^{j-h}(H_a-i+1)^{-1}
\,\cdot\,\sigma_a(x)\,\cdot\,
\prod_{i=1}^{j-h}(H_a+i+1)\,=
\\
&\prod_{i=0}^{j+h}(H_a+i+1)
\,\cdot\,\sigma_a(x)\,\cdot\, 
\prod_{i=0}^{j+h}(H_a-i+1)^{-1}
\end{align*}
while the double coset
$\zec_a(x)\in\Ab\ts/(\n\Ab+\Ab\n')$ is that of the element~of~$\Ab$
\begin{align*}
&\prod_{i=0}^{j-h}(H_a-i+1)^{-1}
\,\cdot\,\sigma_a(x)\,\cdot\,
\prod_{i=0}^{j-h}(H_a+i+1)\,=
\\
\label{qprod}
&\prod_{i=1}^{j+h}(H_a+i+1)
\,\cdot\,\sigma_a(x)\,\cdot\, 
\prod_{i=1}^{j+h}(H_a-i+1)^{-1}\,.
\end{align*}
As a corollary we get the relations in $\Ab\ts/(\n\Ab+\Ab\n')$
\begin{equation}
\label{qxiprod}
\xic_a(x)\,(H_a+1)=
(H_a+1)\,\zec_a(x)\,,
\end{equation}
\begin{equation}
\label{qxiprod2}
(\,\xic_a\,\zec_a\ts)(x)=(\,\zec_a\,\xic_a\ts)(x)=x 
\ \operatorname{mod}\ (\n\Ab+\Ab\n')\,.
\end{equation}

%-----------------------------------------------------------------------------

\subsection{}
\label{sec:pf-i-2}

Let $\D_{mn}$ be the ring of left 
differential operators on the Grassmann algebra~$\G_{mn}\,$,
see Subsection~\ref{sec:2.1} for further notation.
Choose the algebra $\A$ from the previous subsection as
$$
\A=\U(\gl_m)\otimes \D_{mn}\,.
$$
The assignment $E_{ab}\mapsto D_{ab}$ defines a homomorphism
$\phi\colon\U(\gl_m)\to\D_{mn}\,$.
The algebra $\A$ contains a diagonally embedded subalgebra $\U(\gl_m)$ 
generated by elements $X\otimes 1+1\otimes\phi(X)$ 
where $X\in\gl_m\,$. The subspace $1\otimes \D_{mn}\subs\A$ 
is invariant under the adjoint action of %the Lie algebra 
$\gl_m$ and can be chosen as the subspace
$\V$ from the previous subsection. 
In particular, the symmetric group $\S_m\subset\GL_m(\CC)$ 
acts on $\D_{mn}$ by permuting the first indices of the anticommuting
variables $x_{ai}$ and the corresponding left derivations 
$\d_{ai}\,$. For any $\aw\in\h^*$ let
$\J_\aw$ and $\Jb_\aw$ be the right ideals of respectively $\A$ and $\Ab$
generated by all elements $E_{aa}-\aw_a$ where
$a=1 \sco m\,$. Let $\I_{\ts\mu}$ and $\Ib_{\ts\mu}$ 
be the left ideals of respectively $\A$ and $\Ab$ 
generated by all elements
$\d_{ai}\,,E_{aa}-\mu_a$ where $a=1\sco m$ and $i=1 \sco n\,$.

Take any pair of weights $\la,\mu\in\h^*$ such that
all coordinates $\nu_a$ of the difference $\nu=\lambda-\mu$ 
are non-negative integers not exceeding $n\,$. 
Consider the quotient vector space
$$
\A^\la_\mu=\A\;/(\,\n\A+\A\n'+\J_\la+\I_{\ts\mu}\,)\,.
$$
One can define an isomorphism of vector spaces 
$$
\iso_\nu:\G_{mn}^{\ts\nu}\to\A^\la_\mu
$$
by mapping any $x\in\G_{mn}^{\ts\nu}$ to the coset of $1\otimes x\in\A\,$.
The vector space $\A^\la_\mu$ comes equipped with a natural structure 
of an $\YY$-module, see \cite[Section 3]{KN1}. Further, let 
$\rho\in\h^*$ be the weight with coordinates $\rho_a=1-a\,$.
%for each $a=1\sco m\,$. 
Then by \cite[Corollary 2.4]{KN1} the composite map
\beq
\label{isophia}
\iso_\nu\,\al_\nu:\,
\Phi^{\lambda-\rho}_{\mu-\rho}
\to 
\A_{\mu}^{\lambda}
\eeq
is an isomorphism of $\YY$-modules, see also the 
beginning of Subsection~\ref{sec:22} above.

In what follows we will also consider the 
the shifted action of the Weyl group $\S_m$ on $\h^*$:
using the weight $\rho$ determined above, 
for any $\si\in\S_m$ we have
$$
\si\circ\la=\si(\la+\rho)-\rho\,.
$$
Note that by regarding the elements of $\U(\h)$ as polynomial functions
on $\h^*$ we then recover the action of $\S_m$ on $\U(\h)$
defined by \eqref{shiftuh}.
 
If $\aw$ is \emph{generic},
that is if $\aw_a-\aw_b\not\in\ZZ$ for all $a\ne b\,$,
then the weight $\mu$ is generic as well.
Then the quotient vector space
$\A^\la_\mu$ can be identified with another quotient vector space
$$
\Ab^\la_\mu = \Ab\;/(\,\n\Ab+\Ab\n'+\Jb_\la+\Ib_{\ts\mu}\,)\,.
$$
Namely, for each $x\in\G_{mn}^{\ts\nu}$ one can then identify the
cosets of $1\otimes x$ in $\A^\la_\mu$ and $\Ab^\la_\mu\,$.  

\begin{prop}
\label{42}
Suppose that the weight $\la\in\h^*$ is generic. 
Then for any $a=1\sco m-1\,$:

\smallskip

(i) the Zhelobenko operator $\xic_a$ on $\Ab\;/(\,\n\Ab+\Ab\n'\ts)$
induces a linear map
\beq
\label{ua}
\xiclm:
\A^\lambda_\mu
\to
\A^{\sigma_a\circ \lambda}_{\sigma_a\circ\mu}\,\textit{;}
\eeq

(ii) the latter map is $\YY$-intertwining\,;

(iii) we have
$$
%\label{4.3.1}
\iso_{\si_a(\nu)}^{-1}\,\xiclm\,\iso_{\nu}=
\sigma_a\cdot\sum_{r=0}^\infty\,
\frac{(-1)^r\E_{a+1,a}^r\E_{a,a+1}^r}
{r!\,(\lambda_a-\lambda_{a+1}+2)_{\ts r}}\ .
$$
\end{prop}

\begin{proof}
Since the weight $\la$ is generic, 
we can identify the quotient vector spaces
$\A^\lambda_\mu$ and $\A^{\sigma_a\circ \lambda}_{\sigma_a\circ\mu}$ 
respectively with
$\Ab^\lambda_\mu$ and $\Ab^{\sigma_a\circ \lambda}_{\sigma_a\circ\mu}$
as above. Using this identification,
the parts (i) and (ii) of the proposition have been proved in
\cite[Section 3]{KN1}. 
Further, the part (iii) is a particular~case of 
\cite[Proposition 3.5 and Corollary 3.6]{KN2}.
Namely see \cite[Equation 3.12]{KN2}.
%and the last displayed formula in \cite[Subsection 3.2]{KN2}.
%can be read as 
%$$
%\xic_a(1\otimes u)=1\otimes\si_a\left(u+\sum_{r\geq 1}
%\frac{(-1)^r\phi(E_{a+1,a}^rE_{a,a+1}^r)(u)}
%{r!(\la_a-\la_{a+1}+1)_r}\right)
%$$
%Here the element $1\otimes u$ in the left hand side of the 
%equality is regarded as a vector in 
%$\A^\lambda_\mu$, and the right hand side as a vector in
%$\A^{\sigma_a\circ \lambda}_{\sigma_a\circ\mu}$. 
%This is precisely the statement (iii).
%Also see the second part of \cite[Equation 8.3]{KO}.
\end{proof}

Let us now fix $\nu=\lambda-\mu$ and consider the map $\xiclm$
as a function of the parameter $\la\in\h^*$.
It immediately follows from the part (iii) of Proposition \ref{42}
that this function
admits an analytical continuation from the generic $\la$ to all weights
$\lambda$ such that 
$$
\lambda_a-\lambda_{a+1}+1\not=-1,-2,\ldots\,.
$$
Using the isomorphisms \eqref{isophia} and then changing notation
from $\la+\rho$ and $\mu+\rho$ to $\la$ and $\mu$ respectively, we now
obtain an intertwining operator of standard polynomial $\YY$-modules
$$
I_a:\,
\Phi_\mu^\lambda
\to
\Phi_{\sigma_a(\mu)}^{\sigma_a(\lambda)}
$$ 
for any pair of weights $\la,\mu\in\h^*$ such that
all coordinates $\nu_a$ of the difference $\nu=\lambda-\mu$ 
are non-negative integers not exceeding $n\,$ while
\begin{equation}
\label{ladom}
\lambda_a-\lambda_{a+1}\not=-1,-2,\ldots\,.
\end{equation}
Moreover, by the part (iii) of Proposition \ref{42} 
we have an explicit formula for this operator:
\begin{equation}
\label{4.3.2}
I_a\,=\,\al_{\sigma_a(\nu)}^{\,-1}
\ts\biggl(
\sigma_a\cdot\sum_{r=0}^\infty\,
\frac{(-1)^r\E_{a+1,a}^r\E_{a,a+1}^r}
{r!\,(\lambda_a-\lambda_{a+1}+1)_{\ts r}}
\ts\biggr)\,
\al_\nu\,.
\end{equation}

%-----------------------------------------------------------------------------

\subsection{}

Proposition \ref{42} and its implications as described above
have their counterparts for the Zhelobenko operator 
$\zec_a$ instead of $\xic_a\,$. We give them here.
%We will give them in the current subsection. 
Take again any weights $\la,\mu\in\h^*$ such that
all coordinates $\nu_a$ of %the difference 
$\nu=\lambda-\mu$ 
are non-negative integers not exceeding~$n\,$.

\begin{prop}
\label{43}
Suppose that the weight $\la\in\h^*$ is generic. 
Then for any $a=1\sco m-1\,$:

\smallskip

(i) the Zhelobenko operator $\zec_a$ on $\Ab\;/(\,\n\Ab+\Ab\n'\ts)$
induces a linear map
$$
\zeclm:
\A^\lambda_\mu
\to
\A^{\sigma_a\circ \lambda}_{\sigma_a\circ\mu}\,\textit{;}
$$

(ii) the latter map is $\YY$-intertwining\,;

(iii) we have
$$
%\label{4.3.1}
\iso_{\si_a(\nu)}^{-1}\,\zeclm\,\iso_{\nu}=
\sigma_a\cdot\sum_{r=0}^\infty\,
\frac{(-1)^r\E_{a,a+1}^r\E_{a+1,a}^r}
{r!\,(\mu_a-\mu_{a+1}+2)_{\ts r}}\ .
$$
\end{prop}

\begin{proof}
Firstly we make a general observation.
In the setting of Subsection \ref{sec:pf-i-1}
suppose that the algebra $\A$ admits an involutive anti-automorphism 
such that its restriction to
$\gl_m\subset\A$ coincides with the matrix transposition
$E_{ab}\mapsto E_{ba}$ for $a,b=1\sco m\ts$. Then the Zhelobenko operators 
$\xic_a$ and $\zec_a$ on $\Ab/(\n\Ab+\Ab\n'\ts)$ 
are conjugate to each other by the 
involutive linear operator on $\Ab/(\n\Ab+\Ab\n'\ts)$
induced by this anti-automorphism. 
%\begin{equation}
%\label{zec}\zec_a=\theta\,\xic_a\,\theta\,.
%\end{equation}

Now choose $\A=\U(\gl_m)\otimes \D_{mn}$ as in Subsection \ref{sec:pf-i-2}.
Then the matrix transposition on $\gl_m$ extends
to an involutive anti-automorphism of $\A$ which maps
$x_{ai}\mapsto\partial_{ai}$ for $a=1\sco m$ and $i=1\sco n\ts$.
Using our observation with this choice of 
the involutive anti-automorphism of $\A\,$,
the part (i) follows from another property of $\xic_a$ stated below.
Let $\Jpb_\lambda$ be the left ideal of $\Ab$ generated by all elements
$E_{aa}-\aw_a$ where $a=1\sco m\,$. Let
$\Ipb_{\ts\mu}$ be the right ideal of 
$\Ab$ generated by all elements $x_{ai}\,,E_{aa}-\mu_a$
where $a=1\sco m$ and $i=1\sco n\,$. Consider the 
quotient vector space
$$
\Bb^{\ts\la}_{\ts\mu}=\Ab\;/(\,\n\Ab+\Ab\n'+\Jpb_\la+\Ipb_{\ts\mu}\,)\,.
$$
Then the Zhelobenko operator $\xic_a$ on $\Ab\;/(\,\n\Ab+\Ab\n'\ts)$
induces a linear map
$
\Bb^{\ts\lambda}_{\ts\mu}
\to
\Bb^{\,\sigma_a\circ\lambda}_{\,\sigma_a\circ\mu}\,.
$
This property of $\xic_a$
can be proved in the same way as the part (i) of Proposition~\ref{42}. 

Further, since the weight $\la$ is generic we can use Proposition \ref{42}
and define the operator \eqref{ua}.
It then readily follows from the relation \eqref{qxiprod} that
\beq
\label{4.3.66}
(\mu_a-\mu_{a+1}+1)\,U_a = (\la_a-\la_{a+1}+1)\,V_a\,.
\eeq
In particular, the part (ii) of Proposition \ref{43} follows from
the part (ii) of Proposition \ref{42}. 
The part (iii) of Proposition \ref{43}
can be derived from the respective part of Proposition~\ref{42}
by using the above chosen involutive anti-automorphism of the algebra $\A\,$.
%The calcultions of \cite[Proposition 3.5 and Corollary 3.6]{KN2} 
%can be literally repeated for the Zhelobenko maps  $\zec_a$. 
%In the notations used in the proof of Proposition 4.2 the result is
%$$
%\zec_a(1\otimes u)=1\otimes\si_a\left(u+\sum_{r\geq 1}
%\frac{(-1)^r\phi(E_{a,a+1}^rE_{a+1,a}^r)(u)}
%{r!(\mu_a-\mu_{a+1}+1)_r}\right).$$
%The change of parameters in denominators is due to the 
%different placement (right instead of the left)
%of the Cartan element in two kinds of Zhelobenko maps. 
%This is precisely the statement (iii).
%See also \cite[the first part of eq.(8.3)]{KO}.
\end{proof}

Let us again fix $\nu=\lambda-\mu$ and consider the map $\zeclm$
as a function of the parameter $\la\,$.
It immediately follows from the part (iii) of Proposition \ref{43}
that this function
admits an analytical continuation from the generic $\la$ to all weights
$\lambda$ such that 
$$
\mu_a-\mu_{a+1}+1\not=-1,-2,\ldots\,.
$$
Using the isomorphisms \eqref{isophia} and then changing notation
from $\la+\rho$ and $\mu+\rho$ to $\la$~and~$\mu$ respectively, the
$\zeclm$ yields an intertwining operator of standard polynomial 
$\YY$-modules
$$
J_a:\,
\Phi_\mu^\lambda
\to
\Phi_{\sigma_a(\mu)}^{\sigma_a(\lambda)}
$$ 
for any pair of weights $\la,\mu\in\h^*$ such that
all coordinates $\nu_a$ of the difference $\nu=\lambda-\mu$ 
are non-negative integers not exceeding $n\,$ while
\beq
\label{mudom}
\mu_a-\mu_{a+1}\not=-1,-2,\ldots\,.
\eeq
Moreover, by the part (iii) of Proposition \ref{43} 
we have an explicit formula for this map:
$$
J_a\,=\,\al_{\sigma_a(\nu)}^{\,-1}
\ts\biggl(
\sigma_a\cdot\sum_{r=0}^\infty\,
\frac{(-1)^r\E_{a,a+1}^r\E_{a+1,a}^r}
{r!\,(\mu_a-\mu_{a+1}+1)_{\ts r}}
\ts\biggr)\,
\al_\nu\,.
$$

Let us compare $J_a$ with the intertwining operator $I_a$
defined from the previous subsection. Suppose that the weights
$\la$ and $\mu$ satisfy the conditions \eqref{ladom} and \eqref{mudom}
respectively, so that both operators $I_a$ and $J_a$ are defined.
It then immediately follows from \eqref{4.3.66} that
\beq
\label{4.3.6}
(\mu_a-\mu_{a+1})\,I_a = (\la_a-\la_{a+1})\,J_a\,.
\eeq
Here we have taken into account the change from 
$\la+\rho$ and $\mu+\rho$ to $\la$~and~$\mu\,$.
The relation \eqref{4.3.6} shows
that the operator $I_a$ vanishes on the hyperplane $\la_a=\la_{a+1}$ 
in $\h^*$ while the operator
$J_a$ vanishes on the hyperplane $\mu_a=\mu_{a+1}\,$. 
Thus %in the region $\mu_a\ne\mu_{a+1}$ 
the operator $J_a$ can be regarded as a renormalization of the operator
$I_a$ that is also regular on the hyperplane $\la_a=\la_{a+1}$.

Now suppose that the weight
$\la$ still satisfies the condition \eqref{ladom} while
the weight $\mu$ obeys  
\beq
\label{mudominv}
\mu_a-\mu_{a+1}\not=1,2,\ldots
\eeq
instead of \eqref{mudom}. Then by using Proposition \ref{43} 
and the subsequent argument we can define instead of $J_a$ an intertwining
operator $\Phi_{\sigma_a(\mu)}^{\sigma_a(\lambda)}\to\Phi_\mu^\lambda\,$.
Denote it by $J_a^{\,\prime}\,$. The relation \eqref{qxiprod2}~now 
implies the equalities
\beq
\label{ijid}
J_a^{\,\prime}\,I_a=\Id
\quad\text{and}\quad
I_a\,J_a^{\,\prime}=\Id
\eeq
on $\Phi_\mu^\lambda$ and $\Phi_{\sigma_a(\mu)}^{\sigma_a(\lambda)}$
respectively. In particular, under \eqref{mudominv}
the operator $I_a$ is invertible.

%-----------------------------------------------------------------------------

\subsection{}
\label{sec:pf-i-3}

In this subsection we will complete the proof of
Proposition~\ref{prop-i}. Let again $\lambda,\mu\in\h^*$ be 
any weights satisfying all conditions of that proposition.
Choose any reduced decomposition $\si_0=\si_{a_1}\ldots\si_{a_\ell}$ 
of the longest element of $\S_m\,$. Here $\ell=m\ts(m-1)/2\,$.
Using Proposition \ref{42} and the subsequent argument we can define 
for each $s=1\sco\ell$ an intertwining operator
\beq
\label{ias}
I_{a_s}:\,
\Phi_{\,\si(\mu)}^{\,\si(\la)}
\to
\Phi_{\,\si_{a_s}\si(\mu)}^{\,\si_{a_s}\si(\la)}
\quad\ \text{where}\ \quad
\si=\sigma_{a_{s+1}}\ldots\sigma_{a_\ell}\,.
\eeq
Let
$
I=I_{a_1}\ldots I_{a_\ell}
$ 
be the composition of these operators.
We get an intertwining operator
$$
I:\, \
\Phi_\mu^\lambda
\,\to\,
\Phi_{\sigma_0(\mu)}^{\sigma_0(\lambda)}
$$
which does not depend on the choice of reduced decomposition of $\si_0$
due to the braid group relations \eqref{braid3},\eqref{braid2}. 
Moreover by using \eqref{4.3.2} repeatedly, we get an explicit formula 
for the latter operator: 
in the notation of Subsections \ref{sec:2.1} and \ref{sec:22}
$$
I\,=\,
\al_{\sigma_0(\nu)}^{\,-1}\, 
\sigma_0\,Z\,\al_{\nu}
$$
where
\beq
\label{xabla}
Z\ =
\prod_{1\le a<b\le m}^{\longrightarrow}\!
X_{ab}^\la
\eeq
and the factors $X_{ab}^\la$ are ordered 
so that for every $s=1\sco\ell$ the $s\,$th factor from the left
has the indices $a=\si^{-1}(a_s)$ and $\,b=\si^{-1}(a_s+1)\,$.
Here we use the permutation $\si$ from~\eqref{ias}.

Now take the vector \eqref{vlm}.
Due to \cite[Proposition 3.7]{KN1} the vector 
$$
I(\vlm)\in\Phi_{\sigma_0(\mu)}^{\sigma_0(\lambda)}
$$ 
is equal to the vector \eqref{vlms} multiplied by $(-1)^N$
in the notation \eqref{N}, and by the product
$$
\prod\limits_{\substack{1\le a<b\le m \\ \!\!\!\nu_a<\nu_b}}\,
\displaystyle
\frac{\,\la_a-\la_b}{\,\mu_a-\mu_b}\ .
$$
Let us now replace every factor $X_{ab}^\la$ with $\nu_a<\nu_b$ 
in \eqref{xabla} by the corresponding factor $Y_{ab}^\mu\,$.
Here $\mu_a-\mu_b>\la_a-\la_b$ so that 
$
\mu_a-\mu_b\neq0,-1,-2,\ldots
$
and the operator $Y_{ab}^\mu$ is defined. 

These replacements change the product \eqref{xabla} to \eqref{z}.
On the other hand, due to %the relation 
\eqref{4.3.6} for each such replacement
$$
(\mu_a-\mu_b)\,X_{ab}^\la=(\la_a-\la_b)\,Y_{ab}^\mu\,.
$$
It now follows that the operator $I^\la_\mu$ defined by \eqref{i}
satisfies \eqref{ivlm}. This argument completes
the proof of the parts (i,ii,iii) of Proposition~\ref{prop-i}.
Further, due to \eqref{ivlm} our operator $I^\la_\mu$
coincides with the intertwining operator $I(\mu-\rho)$ 
in the notation of \cite{KNP}. Therefore the last part (iv) of
our Proposition~\ref{prop-i} follows directly~from 
\cite[Theorem 1.1~and~Proposition 2.9]{KNP}. 

\begin{remark}
Due to \eqref{4.3.6} any factor $X_{ab}^\la$ with $\nu_a=\nu_b$ 
equals $Y_{ab}^\mu\,$. Hence the strict inequality in the second 
line of the definition \eqref{z} can be replaced by a non-strict one.
\end{remark}

\begin{remark}
The proof of \cite[Proposition 2.9]{KNP} has been based on 
the equalities \eqref{ijid} which hold under the conditions
\eqref{ladom} and \eqref{mudominv}. In the present paper
we gave an independent proof of these two equalities,
by using the properties of Zhelobenko operators.  
\end{remark}

%-----------------------------------------------------------------------------

\subsection{}
\label{sec:pf-ri}
In this subsection we will prove Theorem \ref{prop-ri}.
Suppose the weights $\la$ and $\nu=\la-\mu$ satisfy all conditions
of the theorem. Determine the sequence of signs $\sign$ together with the
weights $\lap$ and $\nup$ as in 
Subsection~\ref{sec:rat}. Then we can define an isomorphism
of $\YY$-modules
\beq
\label{4.5.6}
J^{\ts\la}_\mu:\,\Phi_\mu^\la
\,\to
\Phi_\mu^{\ts\lap}
\otimes
\bigl(\,\underset{a\,:\,\nu_a<0}{\otimes}\,\Delta'_{\ts\mu_a}\,\bigr)
\eeq
as follows. Consider the bijective linear maps
$$
\al_{|\nu|}:\Phi_\mu^\la\to\G_{mn}^{\ts|\nu|}
\ \quad\text{and}\ \quad
\al_{\nup}:\,\Phi_\mu^{\ts\lap}\,\to\,\G_{mn}^{\ts\nup}\,.
$$
Using the notation \eqref{qp} let $S_\sign$ be 
the involutive automorphism of the algebra $\D_{mn}$ such that
$$
S_\sign(x_{ai})=q_{ai}
\quad\text{and}\quad
S_\sign(\d_{ai})=p_{ai}
\quad\text{for}\quad
a=1\sco m 
\quad\text{and}\quad
i=1\sco n\,.
$$
For every $x\in\G_{mn}$ put 
\beq
\label{rep}
R_{\ts\sign}(x)=S_\sign(x)
\prod^{\longrightarrow}_{a\,:\,\nu_a<0}\,
(\ts x_{a1}\ldots x_{an}\ts)
\eeq
where the factors corresponding to the indices $a$
with $\nu_a<0$ are ordered from left to right
as the indices increase. Then 
$$
R_{\ts\sign}:\G_{mn}^{\ts|\nu|}\to\G_{mn}^{\ts\nup}
$$
and moreover by \eqref{rs} for any operator $Y\in\D_{mn}$ 
we have the relation
\beq
\label{rsep}
R_{\ts\sign}\ts(\ts Y(x))=S_\sign(Y)(R_{\ts\sign}(x))\,.
\eeq
Subsection~\ref{sec:pf-mod} shows that
an isomorphism \eqref{4.5.6} can be defined by mapping~any  
$w\in\Phi_\mu^\la$~to
$$
\al_{\nup}^{\ts-1}(R_{\ts\sign}(\al_{|\nu|}(w)))
\otimes
\bigl(\,\underset{a\,:\,\nu_a<0}{\otimes}\,v_n\,\bigr)
$$
where according to \eqref{vd}
$$
v_n=e_1\sw e_n\in\Delta'_{\ts\mu_a}\,.
$$

Similarly, we can define an isomorphism of $\YY$-modules
$$
J^{\,\si_0(\la)}_{\si_0(\mu)}:\,
\Phi^{\ts\si_0(\la)}_{\ts\si_0(\mu)}
\,\to
\Phi^{\ts\si_0(\lap)}_{\ts\si_0(\mu)}
\otimes
\bigl(\,\underset{a\,:\,\nu_a<0}{\otimes}\,\Delta'_{\ts\mu_a}\,\bigr)
\,.
$$
Note that in the latter definition the order of the tensor factors 
$\Delta'_{\ts\mu_a}$ is chosen to be the same as in \eqref{4.5.6}
by using Proposition~\ref{prop-mod}(iv). But the $\YY$-module
$\Phi_\mu^{\ts\lap}$ is polynomial~while~the weights
$\lap$~and~$\nup$ satisfy the conditions of Proposition~\ref{prop-i}.
We get an intertwining operator
$$
I_\mu^{\ts\lap}\,:
\Phi_\mu^{\ts\lap}
\to
\Phi^{\ts\si_0(\lap)}_{\ts\si_0(\mu)}\,.
$$
Take the composition
\beq
\label{composit}
(J_\mu^{\ts\lap}\ts)^{-1}\,
(\ts I_\mu^{\ts\lap}\otimes\Id\ts)\,
J^{\ts\la}_\mu:\,
\Phi_\mu^\la\to
\Phi^{\ts\si_0(\la)}_{\ts\si_0(\mu)}\,.
\eeq
By Proposition~\ref{prop-i} the composite operator
has the properties (i-iv) from Theorem~\ref{prop-ri}.
Put
$$
K\ =\sum_{a<b\,:\,\nu_a<0}\,\nu_a\,\nu_b
\qquad\text{and}\qquad
L\ =\sum_{a>b\,:\,\nu_a<0}\,\nu_a\,\nu_b\,.
$$
Then by the definitions \eqref{N},\eqref{Nb} and \eqref{rep} we 
have the operator relation 
$$
R_{\ts\si_0(\sign)}\,\si_0=(-1)^{\,K+L+N+\Nb}\ts\si_0\,R_\sign\,.
$$
By using the definitions of $J^\la_\mu$ and $J_\mu^{\ts\lap}$ 
along this relation, the operator \eqref{composit} equals
\begin{gather*}
(-1)^N\,
(\ts\al_{\si_0(\nup)}^{\ts-1}\,R_{\ts\si_0(\sign)}\,\al_{\si_0|\nu|})^{-1}
(\ts\al_{\sigma_0(\nup)}^{\,-1}\,\sigma_0\,Z^{\ts\lap}_\mu\,\al_{\nup})\,
(\ts\al_{\nup}^{\ts-1}\,R_{\ts\sign}\,\al_{|\nu|})=
\\[2pt]
(-1)^{\,K+L+\Nb}\,\al_{\si_0|\nu|}^{-1}\,\sigma_0\,
R_{\ts\sign}^{\ts-1}\ts Z^{\ts\lap}_\mu\,R_{\ts\sign}\,
\al_{|\nu|}=
(-1)^{\,K+L+\Nb}\,\al_{\si_0|\nu|}^{-1}\,\sigma_0\,
S_{\ts\sign}(Z^{\ts\lap}_\mu)\,\al_{|\nu|}=
(-1)^{\,K+L}\,
I_\mu^\la\,;
\end{gather*}
see also the definitions \eqref{i},\eqref{ri} and the relation \eqref{rsep}. 
Thus we have proved Theorem~\ref{prop-ri}.

Furthermore, put
$$
M\ =\sum_{a\,:\,\nu_a<0}\,\nu_a\ts(\nu_a-1)/2\,.
$$
Then by the definition \eqref{rep} we have the equality
$$
J^{\ts\la}_\mu\bigl(\vlm\ts\bigr)=(-1)^{\,K+M}\,
v^{\ts\lap}_{\ts\mu}
\otimes
\bigl(\,\underset{a\,:\,\nu_a<0}{\otimes}\,v_n\,\bigr)\,.
$$
Similarly, we have
$$
J^{\,\si_0(\la)}_{\si_0(\mu)}
\bigl(v_{\ts\sigma_0(\mu)}^{\ts\sigma_0(\la)}\ts\bigr)=(-1)^{\,L+M}\,\ts
v_{\ts\sigma_0(\mu)}^{\ts\sigma_0(\lap)}
\otimes
\bigl(\,\underset{a\,:\,\nu_a<0}{\otimes}\,v_n\,\bigr)\,.
$$
We also have
$$
I_\mu^{\ts\lap}\bigl(v^{\ts\lap}_{\ts\mu}\ts\bigr)=
v_{\ts\sigma_0(\mu)}^{\ts\sigma_0(\lap)}\,.
$$
The last three displayed equalities imply that
the operator $I_\mu^\la$ also has the property \eqref{ivlm}.

%-----------------------------------------------------------------------------

\section*{Acknowledgements}

\noindent
%We are grateful to Grigori Olshanski for helpful duscussions.
The first author has been 
supported by the RFBR grant 11-01-00962, 
the joint RFBR grant 11-01-92612-KO
with the Royal Society,
and by the joint grant 11-02-90453-Ukr.
The second author has been supported by the EPSRC grant EP/\ts I\ts014071.
The third author has been supported by
the NSF grant DGE\,-1106400, 
the FASI grant 14.740.11.0347,
the RFBR grant 12-01-33071-mol-a-ved, 
and by the joint RFBR-CNRS grant 11-01-93105-NCNIL\,-a.

\newpage%%%%%%%%%%%%%%%%%%%%%%%%%%%%%%%%%%%%%%%%%%%%%%%%%%%%%%%%%%%%%%%%%%%%%%

%=============================================================================

%=============================================================================

\end{document}